\documentclass[12pt,a4paper,twoside]{amsart}
\usepackage{amsmath}
\usepackage{amssymb}
\usepackage{amsfonts}
\usepackage{a4}

\usepackage{comment}
\usepackage[usenames]{color}

\newcommand{\mC}{{\mathbb C}}

\newcommand{\mP}{\mathbb P}

\newcommand{\mZ}{{\mathbb Z}}

\newcommand{\bO}{\Omega}
\newcommand{\bo}{\omega}

\newcommand{\kk}{\kappa}

\newcommand{\mcC}{\mathcal C}

\newcommand{\mcH}{\mathcal H}

\newcommand{\mcO}{\mathcal O}

\newcommand{\mcS}{\mathcal S}
\newcommand{\mcT}{\mathcal T}
\newcommand{\mcU}{\mathcal U}

\newcommand{\mcZ}{\mathcal Z}

\usepackage{psfrag,graphicx,}
\usepackage[off]{auto-pst-pdf}  

\newcommand{\quickop}[1]{%
\expandafter\DeclareMathOperator\csname #1\endcsname{#1}}

\quickop{Tan}

\quickop{Pf}

\quickop{diag}
\quickop{St}
\quickop{res}
\quickop{supp}
\quickop{stab}
\quickop{End}
\quickop{GL}
\quickop{M}
\quickop{tr}
\quickop{SL}
\quickop{Gal}
\quickop{Lie}
\quickop{Hom}

\quickop{meas}
\quickop{disc}

\newcommand{\dorbit}{{\mathcal{O}}}
\newcommand{\dborbit}{{\mathcal{O}}}

\newcommand{\pmat}[4]{%
        \left(\begin{smallmatrix}
        \vphantom{P^2}#1 & \vphantom{P^2}#2 \\
        \vphantom{P^2}#3 & \vphantom{P^2}#4
        \end{smallmatrix}\right)}

\newcommand{\inv}{^{-1}}

\newcommand{\bb}{{\mathfrak{b}}}

\renewcommand{\gg}{{\mathfrak{g}}}

\newcommand{\ti}{\tilde}
\newcommand{\wt}{\widetilde}

\newcommand{\ov}{\overline}

\newcommand{\wh}{\widehat}

\usepackage[centertags]{amsmath}
\usepackage{amsfonts}
\usepackage{amssymb}
\usepackage{amsthm}
\usepackage{newlfont}
\usepackage{url}

\theoremstyle{plain}
\newtheorem{thm}{Theorem}
\newtheorem{cl}[thm]{Claim}
\newtheorem{cor}[thm]{Corollary}
\newtheorem{conj}[thm]{Conjecture}
\newtheorem{lem}[thm]{Lemma}
\newtheorem{lem*}[thm]{Lemma}
\newtheorem{prop}[thm]{Proposition}
\theoremstyle{definition}
\newtheorem{dfn}{Definition}
\theoremstyle{remark}
\newtheorem{rem}{Remark}
\newtheorem{rem*}{Remark}

\numberwithin{rem}{section} 
\numberwithin{dfn}{section} 
\numberwithin{equation}{section} 
\numberwithin{thm}{section} 

\def\ov{\overline}

\def\smatrix{\smallmatrix}

\def\pmatrix{\left(\smatrix}
\def\endpmatrix{\endsmallmatrix\right)}

\def\!{\operatorname{!}}
\def\wh{\widehat}

\def\C{\mathbb C}

\def\R{\mathbb R}
\def\Z{\mathbb Z}

\def\AA{\mathcal A}
\def\BB{\mathcal B}

\def\NN{\mathcal N}

\def\1{\mathbf 1}
\def\val{\operatorname{val}}
\def\GL{\operatorname{GL}}
\def\PGL{\operatorname{PGL}}
\def\Id{\operatorname{Id}}

\def\SL{\operatorname{SL}}
\def\Hom{\operatorname{Hom}}
\def\chr{\operatorname{char}}
\def\St{\operatorname{St}}
\def\End{\operatorname{End}}
\def\ch{\operatorname{ch}}
\def\tr{\operatorname{tr}}
\def\cusp{\operatorname{cusp}}
\def\ind{\operatorname{ind}}
\def\supp{\operatorname{supp}}
\def\Tsp{A}

\begin{document}

\title[A spectral decomposition of Orbital Integrals]{A spectral decomposition of orbital integrals for $PGL(2,F)$ (with an Appendix by S. Debacker)}

\author{David Kazhdan}

\begin{abstract}
Let $F$ be a local non-archimedian field, $G$ a semisimple $F$-group, $dg$ a Haar measure on $G$ and $\mcS (G)$ be
the space of locally constant complex valued functions $f$ on $G$ with compact support. For any regular elliptic congugacy class $\bO =h^G\subset G$ we denote by $I_\bO$ the $G$-invariant functional on $\mcS (G)$ given by
$$I_\bO (f)=\int _G f(g^{-1}hg)dg$$
This paper provides  the spectral decomposition of functionals
$I_\bO$ in the case  $G=\PGL(2,F)$ and in the last section first steps of such 
an analysis for the  general case.
\end{abstract}
\date{\today}

\maketitle
{\bf Dedicated to A. Beilinson on the occasion of his 60th birthday.}

\medskip

{\it Acknowledgments}. Many thanks for J.Bernstein, S. Debacker and Y. Flicker
who corrected a number of imprecisions in the original draft and S. Debacker for writing an Appendix.

I am partially supported by the ERC grant 669655-HAS.
\section{Introduction} 
Let $F$ be a local non-Archimedean field, $\mcO$ be the ring of integers of $F$, $\cal P\subset\mcO$ the maximal ideal, $k=\mcO /\cal P$ the residue field, $\varpi$ a generator of $\cal P$, $q=|k|$, $\val:F^\times\to\mZ$ the valuation such that $\val(\varpi)=1$ and $\| a\| =q^{-\val (a)}$, $a\in F^\times$, the normalized absolute value. For any analytic $F$-variety $Y$ we
denote by $\mcS (Y)$ the space of locally constant complex-valued functions on $Y$ with compact support and by $\mcS ^\vee (Y)$ the space of {\it distributions} on $Y$.

 Let $G$ be a group of $F$-points of a reductive group over $F, Z$ be the center of $G$, $dz$ a Haar measure on $Z$ and $dg$ a Haar measure on $G$. We denote by $\mcH (G)$ the space of compactly supported measures on $G$ invariant under shifts by some open subgroup. The map $\mapsto fdg$ defines an isomorphisms between the spaces
 $\mcS (G)$ and  $\mcH (G)$.  

We denote  by $\hat G_{cusp}\subset \hat G_2\subset \hat G_t\subset \hat G$ the subsets of 
cuspidal, square-integrable and tempered representations. For any 
$\pi \in \hat G_2$ there exists a notion of the {\it formal degree} $d(\pi ,dg)$
of $\pi$ which depends on a choice of a Haar measure $dg$. We chose this measure in such a way that the formal degree of the {\it Steinberg representation} is equal to $1$ and write
$d(\pi)$ instead of $d(\pi ,dg)$. (See Section $3$ for definitions in the case $G=\PGL(2,F))$.

Given  a regular elliptic conjugacy class $\Omega \subset G$
we denote by $I_\bO$ the functional 
on the space $\mcH (G)$ given by 
 $fdg\mapsto \int _{G/Z}f(ghg^{-1})dg/dz, h\in \bO$
 where $dg/dz$ the invariant measure on  $G/Z$ corresponding to our choice of measures $dg$ and $dz$. 

\begin{rem} The functional $I_\bO$ does not depend on a choice 
of a Haar measure $dg$. In particular these functionals are canonically defined in the case when $G$ is semisimple.
\end{rem}
 We denote by $\hat G$ the set of equivalent classes of smooth irreducible complex representations. For any $\pi \in \hat G$ we denote by $\chi _\pi$ the {\it character} of $\pi$ which the  functional on $\mcH (G)$ given by $\chi _\pi (\mu )=tr (\pi (\mu))$ where 
$\pi (\mu )=\int _G\pi (g)\mu$.

\begin{conj} For any regular elliptic conjugacy class $\Omega \subset G$
there exists  unique  measure $\mu _\bO $ on the subset $\hat G_t\subset \hat G$ 
of tempered representations such that
$$I_\bO  =\int _{\pi \in \hat G}\chi _\pi \mu _\bO.$$
\end{conj}
We say that the measure $\mu _\bO$ gives the {\it spectral description} of the functional $I_\bO$. If $t\in G$ is a regular ellipltic element we will write $I_t:=I_\bO ,\bO =t^G$.

The main goal of this paper is to find the spectral description of the functionals $I_\bO$ in the case $G=\PGL(2,F)$. When the residual characteristic of $F$ is odd  such a description (based on the knowledge of formuals for characters $\chi _\pi$) was given in \cite{ss84}.

We discuss the general of case of  a  general reductive group   in the last Section  $10$, but until  Section $10$ we assume that $G=\PGL(2,F)$.

Let $B\subset G=\PGL (2,F)$ be the subgroup of upper triangular matricies. Then $B=\Tsp U$  where $\Tsp \subset B$ is the subgroup of diagonal and $U\subset B$ of unipotent matrices. We denote by $\Tsp (\mcO )$ the maximal compact subgroup of $\Tsp$ and by $X$ be the group of characters of $\Tsp (\mcO ).$ For any $x\in X$ we define in Section $2$ the notion of  {\it depth} $d(x)\in \mZ _+$  of $x$.

 We denote by $\mcU \subset G$ the subset of non-trivial unipotent elements and by $\nu$ a $G$-invariant measure on $\mcU$.

 For any $x\in X$ we denote by $\rho _x$ 
the representation of $G$ induced from the character $x$ of 
$\Tsp (\mcO)U\subset B$ and define  $\hat G_x\subset \hat G$ as the subset 
of irreducible representations of $G$ which appear 
as subquotients of $\rho _x$. Then
 $\hat G_x=\hat G_{i(x)}, i(x)=x^{-1}$ and we have a decomposition of $\hat G$ in the disjoint union
$$\hat G=\cup _{x\in X/i}\hat G _x\cup \hat G _{cusp}$$
where $\hat G_{cusp}\subset \hat G$ is the subset of cuspidal representations.
This decomposition induces a direct sum decomposition 
$$(\star ) {}\mcS (G)=\oplus _{x\in X/i}\mcS (G)_x\oplus \mcS (G)_{cusp}$$
and the analogous direct sum decomposition of the space $\mcS ^\vee (G)$ of distributions.

Let $\delta ,\bo$ be the distributions on $G$ given by 
$$\delta (f)=f(e), \bo (f)=\int _\mcU f(u)\nu$$ 
We denote by  $\delta _x, \bo _x\in \mcS ^\vee (G)$ the components of 
$\delta$ and $\bo$ in the decomposition $(\star )$ and for $r\geq 0$ define 
$$\delta _r=\sum _{x\in X/i|d(x)\leq r}\delta _x$$
and 
$$\bo _r=\sum _{x\in X/i|d(x)\leq r}\bo _x.$$

In  Section $2$ we define the {\it discriminant} $d(\bO) \in \mZ _+$ of  a regular  elliptic conjugacy class $\bO \subset G$.

\begin{thm}\label{1.1} 
For any elliptic torus $T\subset G$
there exist functions $c_e(t),c_\mcU (t)$ on $T$ 
such that any regular elliptic conjugacy class $\Omega =t^G\subset G, t\in T$ we have an equality 

$$I_t =\sum _{\pi \in \hat G_{cusp}} d(\pi )\chi _\pi (t)+
c_e(t)\delta _{d(\bO )}+c_\mcU (t)\bo _{d(\bO )}$$
of distributions.
\end{thm}
The Plancerel formula \ref{8.1} and Claim \ref{8.2} 
provide spectral descriptions of functionals $\delta _r$ and $\bo _r$ and there the spectral descriptions of $I_r$.
\section{The structure of groups $\Tsp$ and $\PGL(2,F)$} 

For any $r\in\mZ_{\ge 0}$ we define
$U_r\subset\mcO^\times$ by
\[ U_0=\mcO^\times ;\quad U_r=1 +\cal P^r,\quad r>0.\]
We denote by $da$ the Haar measure on $F$ with $\int_\mcO da=1$ and by
$d^\times a$ the Haar measure on $F^\times$ with $\int_{\mcO^\times} d^\times a=1$. Then $d^\times a =(1-q^{-1})^{-1}da/\| a\|$.

We denote by $\Theta$ the group of characters of $F^\times$,
 by $i$ the involution of $\Theta$ given by $\theta\mapsto\theta^{-1}$
 and  by $\Theta_2\subset\Theta$ the subgroup of characters $\theta$ such that $\theta^2=\Id$. We can consider $\theta\in\Theta_2$ as a character of $F^\times /(F^\times)^2$. 

We denote by $\Theta _{un}\subset \Theta$ the subgroup
of unramified characters and write $\Theta _{2,un}=\Theta _2\cap \Theta _{un}$. It is clear that $|\Theta _{2,un}|=2$.

We denote by $X$ the group of characters of $\mcO^\times$ and by  $X_2\subset X$ the subgroup of characters $x$ such that $x^2=\Id$. For any $x\in X$ we denote by $\Theta_x\subset\Theta$
the subset of characters $\theta$ such that $\theta {|\mcO^\times}=x$. For any $x\in X_2$ we define
$$\Theta_{2,x}=\Theta_x\cap\Theta_2.$$
It is clear that $|\Theta _{2 ,un}|=2$ and that 
the group  $\Theta _{2 ,un}$ acts simply transitively on $\Theta _{2 ,x}$ for all $x\in X$. So $|\Theta_{2,x}|= 2$ for all $x\in X$.

It is clear that the map
$$\Theta\to\mC^\times,\qquad\theta\mapsto\theta (\varpi),$$
defines a bijection $\Theta_x\to\mC^\times$ for any $x\in X$. This isomorphism induces a structure of an algebraic variety on $\Theta_x$, for each $x\in X$. We denote by $\mC [\Theta_x]$ the algebra of regular functions on $\Theta_x$ which is isomorphic to $\mC [z,z^{-1}]$.

In the case when $x^2=\Id$ the involution $i$ acts on $\Theta_x$. We
denote by $\mC [\Theta_x/i]\subset\mC [\Theta_x]$ the subring of invariant
functions. It is clear that  $\mC [\Theta_x/i]$ is isomorphic to $\mC [z'], z'=z+z^{-1}$.

\begin{dfn} 
For any $x\in X$ we denote by $d(x)$ the minimal integer $r\ge 0$ such that
the restriction of $x^2$ to $U_r$ is trivial. Thus $x^2|U_{d(x)}=1$ and $x^2|U_{d(x)-1}\not=1$ where subgroups $U_r$ are defined in the beginning of this Section.
\end{dfn}

$$\ti G=\GL(2,F),G'=\{ g\in \GL(2,F)| max _{i,j\in (1,2)}\|g_{ij}\|=1\}$$ 
$\ti p:\ti G\to G=\PGL(2,F)$ be the natural projection, 
and $p$ the restriction of $\ti p$ on $G'$. The map $p:G' \to G$ is surjective and the 
group $\mcO ^\times$ acts simply transitively on fibers of $p$. 
\begin{cl} For any $\mu \in \mcH =\mcH (G)$ there exists unique $\mcO ^\times$-invariant   measure $\ti \mu \in \mcH (\ti G)$ supported on $G'$ such that $p_\star \ti \mu =\mu$.
\end{cl}

We denote by $p^\star :\mcH (G)\to \mcH (\ti G)$ the map $\mu \to \ti \mu$. It is clear that the map $p^\star$ is $G$-equivariant.`

We often describe elements $g\in G=\PGL(2,F)$ in terms of a preimage $\wt g$ in $\GL(2,F)$ under the map $p:\GL(2,F)\to G$ and  matrix coefficients of $\wt g_{ij}$.
For any $g\in G$ the ratio $\frac {\wt g^2_{11}}{\det (\wt g)}$ does not depend
on a choice of a representative $\wt g$. We denote it by $\frac { g^2_{11}}{\det (g)}$.

We denote by $K\subset G$ the image of $\GL(2,\mcO)$ and by $\Tsp$ the image of the group of diagonal matrices. We use the map
$$\begin{pmatrix}a_{11}&0\\0&a_{22}\end{pmatrix}\mapsto a_{11}/a_{22}$$
to identify $\Tsp$ with $F^\times$ and $\Tsp(\mcO)$ with $\mcO^\times$, where
$\Tsp(\mcO)=\Tsp\cap K$. We denote by $\ti \in \GL (2,F)$ 
the matrix 
$$\begin{pmatrix}1&0\\0&\varpi\end{pmatrix}\mapsto a_{11}/a_{22}$$
and by  $t\in \Tsp$ the image of $\ti t$ in $\Tsp$.

The following is well known.

\begin{cl}\label{2.1} 
The subsets $Kt^nK\subset G$, $n\geq 0$, are disjoint, and $G=\cup _{n\geq 0}Kt^nK$.
\end{cl}
We write $G^{\leq m}:=\cup _{0\leq n\leq m}Kt^nK$.

We define $\Gamma_0=K$ and for any $d\geq 1$ we denote by $\Gamma_d\subset K$ the image
 of the subgroup $\ti \Gamma _d$ of matrices $g$ in $\GL(2,\mcO)$ with 
$g_{21}\in\cal P^d$. So $I:=\Gamma_1$ is an {\it Iwahori subgroup} of $G$. We denote by
 $p_1:\ti \Gamma _1\to \Gamma _1$ the restriction of $p$ on $\ti \Gamma _1$.

We denote by $U\subset G$ the image of the subgroup of matrices of the form
$$g_u=\begin{pmatrix}1&u\\0&1\end{pmatrix},$$
write $B=\Tsp U$,
and denote by $b\mapsto\bar b$ the projection $B\to B/U\simeq \Tsp\simeq F^\times$.
 For any $\theta\in\Theta$ we denote by the same letter $\theta$ the character of $B$ given by $b\mapsto\theta (\bar b)$.

We denote by $det{}_2$ the map
$$\det{}_2:G\to F^\times /{(F^\times )}^2,\qquad g\mapsto\det(\wt g){(F^\times )^2}$$
and by $G^0$ the kernel of $\det_2$. If $\chr (F)\neq 2$ then $G^0$ is an open subgroup of $G$.

For any character $x\in X$ such that $d=d(x)>0$
we denote by $\wt x :\Gamma_d\to\mC^\times$ the map
$$g\mapsto x\left(\frac {g^2_{11}}{\det ( g)}\right).$$
If $d(x)=0$, that is $x^2=\Id$, we define a character  $\wt x$ of $\Gamma_0=K$ by
$\wt x(g)=x(\det_2(g))$.
\begin{cl} 
For any $x\in X$ the map  $\wt x$ is a character of $\Gamma_{d(x)}$.
\end{cl} 
\begin{dfn} 
For any regular elliptic conjugacy class $\bO\subset G$ we let $d(\bO)$ be
the biggest number $d$ such that $\bO$ intersects $\Gamma_d$.
\end{dfn}

\section{Basic structure of representations of $G$} 
We say that a measure $\mu$ on $G$ is \emph{smooth} 
if it is $R$-invariant for some open subgroup $R\subset G$.
Let $\mcH$ be the space of complex-valued compactly supported smooth measures $\mu$ on $G$. For any open compact subgroup $R\subset G$ we denote by $\ch_R\subset\mcH$ the normalized Haar measure on $R$.

Convolution, denoted by $\ast$, defines an algebra structure on $\mcH$. The algebra $\mcH$ acts on $\mcS (G)G$ by convolution from the right, $(f,\mu)\mapsto f\ast\mu$, and also from the left.

The group $G$ acts on $\mcH$ by conjugation. We denote by $\mcH_G$ the space of coinvariants which  is equal to the quotient $\mcH /[\mcH ,\mcH]$.

We denote by $\mcC$ the category of smooth complex representations of $G$ and by $\wh G$ the set of equivalence classes of smooth irreducible representations of $G$.

The group $G$ acts on $\mP^1(F)$ and therefore on the spaces $\mcS (\mP^1(F))$. It is clear that the subspace $\mC$ of constant functions invariant  and we obtain the {\it Steinberg} representation $St$ of $G$ on the space 
$\mcS (\mP^1(F))/\mC$. It is well known (see \cite{GGP} )  the representation $St$ of $G$
 is irreducible. 
 For any $\theta\in\Theta_2$ we denote by $\mC_\theta$ the one-dimensional representation $g\mapsto\theta(\det_2(g))$, and define $\St_\theta=\St\otimes\mC_\theta$.

For any $(\pi ,V)\in\mcC$, $\mu\in\mcH$, we define
$$\pi (\mu)=\int_G\pi (g)\mu\in\End(V).$$
For irreducible  representations  $\pi$ of $G$ 
 the operator $\pi (\mu)$ is of finite rank for any $\mu\in\mcH$ and we define the {\it character} $\chi _\pi$ on $G$, as a 
{\it generalized function} (a functional on $\mcH$) by
$$\chi _\pi (\mu)=\tr\pi (\mu),\qquad \mu \in \mcH .$$ 
By \cite{JL}, there exist a locally $L^1$-function on $G$, (that we  denote by $\chi _\pi$) such that
$$\chi _\pi (\mu)=\int _G\chi _\pi\mu$$

We define a map $\kk:\mu\mapsto\wh\mu$ from $\mcH$ to functions on $\wh G$ by
$$\wh\mu (\pi)=\tr\pi (\mu).$$
It is clear that $\kk$ descends to a map from $\mcH_G$ to functions on $\wh G$.

We say that an irreducible representation $(\pi ,V)$ of $G$ is {\it square-integrable} if it is {\it unitarizable} (that is, there exists a nonzero $G$-invariant Hermitian form $(,)$ on $V$),
and for every $v\in V$ the function $g\mapsto (\pi (g) v,v)$ on $G$ belongs to $L^2(G)$. We denote by $\wh G_2\subset\wh G$ the subset of square-integrable representations. Let $dg$ be a Haar measure on $G$.

The following Claim follows from \cite{HC}.
\begin{cl}\label{3.1} 
$a)$ For every $(\pi ,V)\in\wh G_2$ there exists a number $\deg (\pi)=\deg (\pi,dg)>0$, called the ${\operatorname{formal}}$ ${\operatorname{degree}}$ of $\pi$, such that
$$\int_G|m_v(g)|^2dg=\frac{1}{\deg(\pi)}, \qquad m_v(g)=(\pi (g) v,v),$$
for any $v\in V, (v,v)=1$, where $dg$ is a Haar measure on $G$.

$b)$ There exists a unique choice of $dg$ with $\deg (\St,dg)=1$.

$c)$ For any irreducible square-integrable representation $(\pi ,V)$ and any $v\in V, (v,v)=1$,
the sequence of locally constant functions
$$(I_n(v))(g) :=\int _{h\in G^{\leq n}}m_v(hgh^{-1})dh$$
on $G$ converges as a generalized function to the the character
$\chi _\pi/\deg(\pi ,dg).$ In other words, for any $\mu \in \mcH$ the sequence $\{ \int I_n(v)\mu \}$ converges to $\hat \mu (\pi)/\deg( \pi)$.
\end{cl}

For any smooth representation $(\pi ,V)$ of $G$ we denote by $J(V)$ the normalized  {\it Jacquet}) functor which is a representation of $\Tsp$ acting on the space $V/V(U)$ where $V(U)$ is the span of $\{ \pi (u)v-v, u\in U,v\in V\}$.
We define the action of $\Tsp$ on $JV)$ by  $a\mapsto\|a\|^{1/2}\pi (a)$, $a\in \Tsp$, of $\Tsp$ on $V_U$. Here $\| a\|=\|t_1/t_2\|$ for $a$ represented by $\begin{pmatrix}t_1&0\\0&t_2\end{pmatrix}$.

We say that $V$ is {\it cuspidal} if $J(V)=\{ 0\}$.

We denote by $\mcC_{\cusp}$ the subcategory of cuspidal representations and by $\wh G_{\cusp}\subset\wh G$ the
subset of equivalence classes of irreducible cuspidal representations. Since
matrix coefficients of cuspidal representation of $G$ have compact support (see \cite{JL}) we have an inclusion $\wh G_{\cusp}\subset\wh G_2$.

\section{Induced representations} 
For any $\theta\in\Theta$ we denote by $(\pi_\theta ,R_\theta)$ the representation of $G$ unitarily induced from the character $b\mapsto\theta (\bar b)$
of $B$. So $R_\theta$ is the space of locally constant complex valued functions $f$ on $G$ such that
$$f(gb)=\theta (\bar b)\|\bar b\|^{1/2} f(g),\quad g\in G,\quad b\in B,$$
and $G$ acts on $R_\theta$ by left shifts: $(\pi_\theta(x)f)(g)=f(x^{-1}g)$.

Since $G=KB$, the restriction to $K$ identifies the space $R_\theta$ with the space
$R_x$, $x=\theta {|\mcO^\times}$, where $R_x$ is the space
of locally constant functions $f$ on $K$ such that
$$f(kb)=\theta (\bar b) f(k),\quad k\in K,\quad b\in B\cap K.$$
It is clear that in this realization the operator $\pi_\theta (\mu)\in\End (R_x)$
is a regular function
on $\theta \in \Theta_x$ for any $\mu\in\mcH$ and so the function
$$\wh\mu_x :\theta\mapsto\wh\mu (\pi_\theta),\quad\theta\in\Theta_x,$$
belongs to $\mC [\Theta_x]$.

 The following result is well known, see \cite{JL}.

\begin{prop}\label{4.1} 
$a)$ For any $\theta\in\Theta$ we have $\End_G(R_\theta)=\mC$.\\
$b)$ A representation $\pi_\theta$ is reducible
if and only if $\theta (a)=\theta_2(a)\| a\|^{1/2}$ or $\theta (a)=\theta_2(a)\| a\|^{-1/2}$
where $\theta_2\in\Theta_2$.
In the second case $\pi_\theta$ has a one-dimensional subrepresentation $\mC_{\theta_2}$, and the quotient is isomorphic to $\St_{\theta_2}$. In the first case $\pi_\theta$ has $\St_{\theta_2}$ as a subrepresentation and the quotient is isomorphic to $\mC_{\theta_2}$.\\
$c)$ Let $\theta ,\theta'\in\Theta$ be such that $\pi_\theta ,\pi_{\theta'}$ are
irreducible. Then the representations $\pi_\theta ,\pi_{\theta'}$ are isomorphic
iff $\theta'=\theta^{\pm 1}$.\\
$d)$ We have a disjoint union decomposition
$$\wh G =\wh G_2\cup\left(\cup_{\theta_2\in\Theta_2}\mC_{\theta_2}\right)\cup
\left(\cup_{\theta\in (\Theta-\Theta_2)/i}\pi_\theta\right).$$
$e)$ We have a disjoint union decomposition
$$\wh G_2=\wh G_{\cusp}\cup\left(\cup_{\theta_2\in\Theta_2}\St_{\theta_2}\right).$$
$f)$ For any $\theta \in \Theta$ the character $\chi _\theta :=\chi _{\pi _\theta}$ is given by  a locally $L^1$-function on $G$ supported on split elements such that
$$\chi _\theta (\begin{pmatrix}a&0\\0& b\end{pmatrix})=
\frac{\theta (a/b)+\theta (b/a)}{\|(a-b)^2/ab\|^{1/2}}.$$
\end{prop}

\begin{dfn} 
$(1)$ For any $x\in X$, we denote by $(\tau_x, V_x)$ the representation of $G$ by left shifts  on the space of locally constant compactly supported functions $f$ on $G$ such that

$$f(gtu)=x(t)f(g),\quad g\in G,\quad t\in \Tsp(\mcO),\quad u\in U.$$
$(2)$ We denote by $f_0\in V_x$ the function supported on
$\Gamma_dU$, $d=d(x)$, and such that
$$f_0(\gamma u)=\wt x(\gamma),\quad\gamma\in\Gamma_d,\quad u\in U.$$
$(3)$ We denote by $p_\theta$, $\theta\in\Theta_x$, the map
$$p_\theta :V_x\to R_\theta,\quad (p_\theta (f))(g)
=\sum_{n\in\mZ}f(gt ^n)q^n\theta(\varpi^n),$$
where as before $t$ is the image in $G$ of
$$\begin{pmatrix}1&0\\0&\varpi\end{pmatrix}\in\GL(2,F).$$
\end{dfn}

We also have

\begin{prop}\label{4.2} 
If $M\subset V_x$ is a $G$-invariant subspace such that $p_\theta (M)=R_\theta$ for all $\theta\in\Theta_x$ then $M=V_x$.
\end{prop}

\begin{proof}  As follows \cite{Be}  it suffices to show that there is no
nonzero morphism from $V_x/M$ to an irreducible representation of $G$.
But  as follows from \cite{BZ}  all morphisms from $V_x$ to an irreducible representation of $G$ are factorizable through a projection $p_\theta$ for some $\theta \in \Theta_x$.
\end{proof}

\begin{cor}\label{4.3} 
If $x^2\neq Id$ then the function $f_0$ generates $V_x$ as an $\mcH$-module.
\end{cor}

\begin{proof}
It is clear that $f_{\theta ,0} :=p_\theta (f_0)\in R_\theta$ is not equal to $0$. Moreover, it follows from Proposition \ref{4.1} a),b) that it generates $R_\theta$ as an $\mcH$-module. But then Proposition \ref{4.2} implies that $f_0$ generates $V_x$ as an $\mcH$-module.
\end{proof}

The following result follows from Corollary \ref{4.3}.
We assume that $x^2\neq Id$ and use the identification of the ring  $\mC [\Theta]$ with $\mC [z,z^{-1}]$ as in the Introduction. Let 

$$\alpha :\mC[\Theta_x]\simeq\mC [z,z^{-1}]\to\End_G(V_x)$$
be the algebra morphism defined by $((\alpha (z))(f))(g)=q^{-1}f(gt^{-1})$, $f\in V_x$.\\

\begin{cor}\label{4.4} 

$a)$ For any $S\in\End_G(V_x)$, and $\theta\in\Theta_x$, the map $S$ preserves the subspace $\ker (p_\theta)\subset V_x$ and so defines $\wh S(\theta)\in\End_G(R_\theta)=\mC$.\\
$b)$ For any $S\in\End_G(V_x)$, the function $\wh S$ on $\Theta_x$  belongs to
 $\mC [\Theta_x]\simeq\mC [z,z^{-1}]$.\\
$c)$ The maps

$$\End_G(V_x)\to \mC [\Theta_x], S\mapsto\wh S$$
and
$$\mC [\Theta_x]\to \End_G(V_x), s\mapsto\alpha (s)$$
are mutually inverse.
\end{cor}

\section{Structure of the representation $(\tau_x,V_x)$ when $d(x)>0$} 

In this section we fix a character $x\in X$ such that $d(x)>0$ (so $x^2\neq Id$).
\begin{dfn} 
$(1)$ We denote by $\rho_x$ the representation $\ind^G_{\Gamma_d}\wt x$ of $G$
on the space $\wt W_x$ of locally constant functions $\phi$ on $G$ such that
$$\phi (g\gamma)=\wt x(\gamma)\phi (g),\quad g\in G,\quad\gamma\in\Gamma_d,$$
and by $W_x\subset\wt W_x$ the subspace of functions with compact support.\\
$(2)$ Denote by $\phi_0\in W_x$ the function supported on $\Gamma_d$ and equal to $\wt x$ there, and define
$$\mu_0:=\phi_0\ch_{\Gamma_d}\in\mcH.$$
$(3)$ Let $A:V_x\to\wt W_x$, $B:W_x\to V_x$ be the $G$-morphisms defined by
$$A(f)=f\ast\mu_0,\qquad B(\phi)=\phi_U,\qquad \phi_U(g)=\int_ U\phi (gu)du,$$
where $du$ is the Haar measure on $U$ which is normalized by $\int_{U\cap K}du=1$.
\end{dfn}

\begin{lem}\label{5.1} 
$a)$ $B(\phi_0)=f_0.$\\
$b)$ $A(f_0)=\phi_0.$\\
$c)$ $A$ defines an isomorphism $A:V_x\to W_x.$\\

$d)$ $\End_G(V_x)\simeq\End_G(W_x).$
\end{lem}

\begin{proof}
Part a) is clear. It is also clear that the restriction of $A(f_0)$ to $\Gamma_d$ is equal to $\wt x$ and that $\supp (A(f_0))\subset\Gamma_dU\Gamma_d$. So to prove (b) it suffices to check that for any $u\in F,\| u\|>1$, we have
$$\int_{\Gamma_d}f_0(g_u\gamma)\wt x(\gamma)^{-1}d\gamma =0,$$
where $d\gamma$ is the normalized Haar measure on $\Gamma_d$ and
$g_u =\begin{pmatrix}1&u\\0&1\end{pmatrix}.$ To see this, write $\gamma$ as
$\gamma_0\gamma_1$, $\gamma_0=\begin{pmatrix}1&0\\c&1\end{pmatrix}$,
$\gamma_1=\begin{pmatrix}a&b\\0&d\end{pmatrix}$. Note that $\wt{x}(\gamma)=\wt{x}(\gamma_1)=x(a/d)$. The integral equals
\[ \int f_0\left(\begin{pmatrix}1&u\\0&1\end{pmatrix} \begin{pmatrix}1&0\\c&1\end{pmatrix}\gamma_1\right)\wt{x}(\gamma_1)^{-1}dc\,d\gamma_1
=\int f_0\left(\begin{pmatrix}1+uc&0\\c&(1+uc)^{-1}\end{pmatrix}\gamma_1
\begin{pmatrix}1&\frac{d}{a}\frac{u}{1+uc}\\0&1\end{pmatrix}\right)x(a/d)^{-1}dc\,d\gamma_1 \]
\[ =\int x(1+uc)^2dc,\qquad \{c\in\cal P^d;\,1+uc\in\mcO^\times\}, \]
since $f_0$ is supported on $\Gamma_dU$ (so it vanishes unless $|1+uc|=1$)and  we are integrating a nontrivial character  the integral  is 0.

The part $c)$ follows from the parts $a),b)$ since
by Corollary \ref{4.3}`
the function $f_0$ generates $V_x$ as an $\mcH$-module
and ,as easy to check, the function $\phi _0$ generates $W_x$ as an $\mcH$-module. The part $d)$ follows from $c)$.

\end{proof}

\section{Algebras of endomorphisms} 
As before we fix (in this section) a character $x\in X$  such that $x^2\neq 1$.

\begin{lem}\label{6.1} 
Let $\mcH. '_x\subset\mcH$ be the subalgebra of measures $\mu$ with
$$l_\gamma (\mu)=r_\gamma (\mu)=\wt x (\gamma)\mu,\quad\gamma\in\Gamma_d,$$
where $l_\gamma$, $r_\gamma $ are left and right shifts by $\gamma$. Then

 $(1)$ $\mu_0=\phi_0\ch_{\Gamma_d}$ is the unit of $\mcH '_x$.

 $(2)$ Convolution on the right defines an isomorphism $\beta :\mcH '_x\to\End_G(W_x)$.
\end{lem}

\begin{proof}
(1) is clear. 

For (2), note that the map $S\mapsto S(\phi_0)\ch_{\Gamma_d}$ defines a morphism $\ti \beta :\End_G(W_x)$ $\to \mcH '_x$. One checks that the compositions $\beta \circ \ti \beta$ and 
$\ti \beta \circ \beta$ are  the identity maps. So we can identify $\End_G(W_x)$ with $\mcH '_x$. 

\end{proof}

As follows from  from Lemma \ref{5.1} we identify the ring $\End_G(W_x)$ 
with $\End_G(V_x)$ and therefore (by Corollary \ref{4.4} ) with 
the ring $\mC [z,z^{-1}]$.

For any $n\in\mZ$ we denote by $\phi_n\in W_x$ the function supported
on $\Gamma_dt ^n\Gamma_d$ with
$$\phi_n(\gamma't ^n\gamma'')=\wt x(\gamma'\gamma''),\qquad\gamma',\,\,\,\gamma''\in\Gamma_d,$$
and write $\mu_n=\phi_ndg\in \mcH '_x$.

The following result follows from the commutativity of the algebra $\mcH '_x$ and Frobenius reciprocity.

\begin{cl}\label{6.2} 
Let $p:(\rho_x ,W_x)\to (\pi ,W)$ be an irreducible quotient of $W_x$. Then\\
$a)$ The action of $\End_G(W_x)=\wt\mcH_x$ on $W_x$ preserves $\ker (p)$
and therefore induces a homomorphism $\wt p:\mcH '_x\to\End_G(W)=\mC$.\\
$b)$ For any $\mu\in \mcH '_x$ we have
$$\wt p (\mu)=\wh\mu (\pi).$$
\end{cl}

\begin{lem}\label{6.3} 
$a)$ The map $\mu\mapsto\wh\mu (\pi_\theta)$, $\theta\in\Theta_x$,
defines an isomorphism $\mcH '_x\to\mC [\Theta_x]$.\\
$b)$ $\wh\mu_n (\theta)=cz^n,\qquad c\neq 0$ where 
we identify the 
space $\Theta_x$ with $\mC^\times$, $\theta\mapsto z=\theta(\varpi)$.

$c)$ The set $\{\mu_n;\,n\in\mZ\}$ is a basis of the space $\mcH '_x$. 
\end{lem}

\begin{proof} $a)$ The first part follows immediately from Lemma \ref{5.1} and Proposition \ref{6.2}.

$b)$ Let
$$R^0_\theta =\{ f\in R_\theta |\pi_\theta (\gamma)f=\wt x(\gamma ^{-1})f,\quad\gamma\in\Gamma_d\}.$$

It follows from Lemma \ref{5.1} and Claim \ref{6.2} that $\dim (R^0_\theta) =1$ and that this space is equal to $\mC\cdot f_{\theta,0}$, where $f_{\theta,0}$ was defined to be $p_\theta (f_0)$ in the proof of Corollary \ref{4.3}. Since $f_{\theta,0}(e)=1$, it is sufficient to show that
$$((\pi_{\theta }(\mu_n))(f_{\theta,0}))(e)=cz^n,$$
but this is immediate. Since the map $\mu \to \hat \mu$ is not a zero map we see that  $c \neq 0$.

The part $c)$ follows from $b)$.
\end{proof}

\section{Categories of representations} 
Fix $x\in X$. Let $\wh G_x\subset\wh G$ be the set of equivalence classes of irreducible
representations of $G$ which appear as subquotients of $(\tau_x, V_x)$. It can be described as the set of equivalence classes of irreducible subquotients of the representations
$\{ R_\theta\}$ for $\theta\in\Theta_x$. It follows from Proposition \ref{4.1} that the
set $\wh G_x$ depends only on the image of $x$ in $X/i$ and that for distinct $x',x\in X/i$ the sets $\wh G_x$ and $\wh G_{x'}$ are disjoint.

We denote by $\mcC_x\subset\mcC$ the subcategory of representations the equivalence classes of whose irreducible subquotients belong to $\wh G_x$.

The following result is well known (see \cite{BZ} for parts $a)$ and $b)$ and \cite{BDK} for part $c)$).

\begin{prop}\label{7.1} 
$a)$ We have a decomposition
$$(1) \mcC =\mcC_{\cusp}\oplus\left(\oplus_{x\in X/i}\mcC_x\right).$$
This decomposition defines the direct sum decompositions
$$(2) \mcS (G) =\mcS (G)_{\cusp}\oplus\left(\oplus_{x\in X/i}\mcS (G)_x\right)$$
and
$$(3){}\mcS ^\vee (G) =\mcS ^\vee (G)_{\cusp}\oplus\left(\oplus_{x\in X/i}\mcS ^\vee(G)_x\right)$$
and
$$(4) \mcH_G=\mcH_{G,\cusp}\oplus\left(\oplus_{x\in X/i}\mcH_{G,x}\right).$$
$b)$ For any $x\in X$, $x^2\neq\Id$, and $\mu\in\mcH_{G,x}$, the function $\wh\mu$ is supported on $\Theta_x$, and the map $\kk :\mu\mapsto\wh\mu$
defines an isomorphism from $\mcH_{G,x}$ to $\mC[\Theta_x]$.\\
$c)$ For any $x\in X_2$, $\mu\in\mcH_{G,x}$, the function $\wh\mu$ is supported on
$$(\Theta_x/i)\cup\left(\cup_{\theta\in\Theta_{2,x}}\St_\theta\right),$$
and the map $\kk :\mu\mapsto\wh\mu$
defines an isomorphism from $\mcH_{G,x}$ to
$$\mC[\Theta_x/i]\oplus\left(\oplus_{\theta\in\Theta_{2,x}}\mC_\theta\right).$$
\end{prop}
\begin{lem}\label{7.2} 
$a)$ For $x\in X-X_2$ we have $\mcH '_x\subset\mcH_x$.

$b)$ The  map $\mcH 'x\to \mcH _{G,x}$ is an isomorphism.
\end{lem}

\begin{proof}
$a)$ Let $(\pi ,V)$ be an irreducible representation such that $\pi (\mu)\neq 0$ for some
$\mu\in\wt\mcH_x$. We want to show that $\pi\in\wh G_x$.

Since $\pi (\mu)\neq 0$ we see that $V^0\neq\{ 0\}$, where
$$V^0=\{ v\in V|\pi (\gamma)v=\wt x(\gamma ^{-1})v,\quad\gamma\in\Gamma_d\}.$$

Since $V|\Gamma_d=V^0\neq\{ 0\}$,

By the Frobenius reciprocity we have
$$\Hom_G(W_x,V)=\Hom_{\Gamma_d}(\wt{x},V)=V_0$$. Therefore
 $V$ is a quotient of $W_x$. So the lemma follows from Lemma \ref{5.1} c) which asserts the equivalence of the representations $V_x$ and $W_x$ of $G$.

$b)$ follows now from \cite{BDK}.
\end{proof}

\begin{cor}\label{7.3} 
Let $\bO \subset G$ be a regular elliptic conjugacy class
and  $x\in X$ be such that  $d(x)>d(\bO)$. Then  $I_\bO (\mu)=0$ for  any $\mu =fdg \in \mcH '_x$ where

$$I_\bO (fdg)=\int_G f(ghg^{-1})dg,\qquad h\in\Omega.$$

\end{cor}

\begin{proof} Since $d=d(x)>0$ it follows from  Lemma \ref{6.3} b) that the set $\{\mu_n=\phi _ndg;\,n\in\mZ\}$ is a basis of the space $\wt\mcH_x$ which (by Lemma \ref{6.3} a) and Proposition \ref{7.1} b)) is isomorphic to $\mcH_{G,x}$. So it suffices to check that  $\int_G\phi_n(ghg^{-1})dg=0$ for all $n\in\mZ$. If $n\not =0$ 
then all  elements of $\Gamma_dt^n\Gamma _d$ are split and so the support of $\phi _n$ is disjoint from $\bO$. On the other hand if $n=0$ then $\Omega\cap\Gamma_d=\emptyset$ by definition of $d(\bO)$.
\end{proof}

\section{The Plancherel formula} 
Consider the distribution  $\delta $ onn $\mcS (G) ,\delta (f):=f(e)$. The part $a(1)$ of Proposition \ref{7.1} implies the decomposition
$$\delta =\delta _{cusp}+ \sum _{x\in X/i}\delta _x$$
where 
$$\delta _{cusp}\in \mcS ^\vee (G)_{\cusp},
 \delta _c\in \mcS ^\vee (G)_x$$

The {\it Plancherel fomula} (see \cite{AP}) describes the functionals $\delta _x$ and $\delta _{cusp}$. Let $S^1=\{ z\in Z| \|z \| =1\}$ and $|dz|$ the Haar measure on $S^1$ such that $\int _{S^1}|dz|=1$. I will use notations of the section $1$ and in particular the identification $z\to x^z\in \Theta _x$ of
$\mC ^\times$ with  $\Theta _x$.
\begin{prop}\label{8.1} 
$a)$ $\delta _{cusp} =\sum _{\pi \in \hat G_{cusp}}d(\pi ,dg)chi _\pi$.

$b)$ If $x^2=Id$ then
$$\delta _x=
\sum _{\theta \in \Theta _{2,x}}\chi _{St_\theta} +\int _{z\in S^1}\frac {|(z-1)(z^{-1}-1)|^2}{|(z/q-1)(z^{-1}/q-1)|^2}\chi _{x^z}|dz|$$
$c)$ If  $x^2\neq Id$ then
$$\delta _x=\int _{z\in S^1}\gamma (x)
\chi _{x^z}|dz|$$
where $\gamma (x)$ are explicit constants (see \cite{AP}).
\end{prop}
Let $dt$ be the Haar measure on $A$ such that $\int _{A(\mcO )}dt=1$ and $dg/dt$ the corresponding $G$-invariant measure on $G/A$. For any $s\in F-\{ 0,1\}$ we write
$$a_s=\begin{pmatrix}s&0\\0&1\end{pmatrix}\in A$$

and  denote by $\bo _s$ the functional  on $\mcS (G)$ given by
$$\bo _s(f)=|s-1|\int _{\ti g\in G/A}f(\ti ga_z\ti g^{-1})dg/dt$$

Let $\mcU\subset G$ be the set of
regular unipotent elements. Since $G$ acts transitively on $\mcU$
and the stationary subgroup is unimodular (it actually is isomorphic to $U\subset G$), there exists a unique (up to a scalar) $G$-invariant measure $\nu$ on $\mcU$.

The following claim is well known and is an easy exircise.

\begin{cl}\label{8.2} 
$a)$ For any $f\in \mcS (G)$ the integral
$$\bo (f):=\int_{G/U}f(g\pmatrix 1&1\\ 0&1\endpmatrix g^{-1})\ov {dg}$$

is absolutely convergent.

$b)$  One can choose a $G$-invariant measure $nu$ on $\mcU$ in such a way that $\bo (f) \equiv \lim _{s\to 1}\bo _s(f)$
\end{cl}

We define
Let $\bo _{s,x}$ be the components of $\bo$ in the decomposition of Proposition \ref{7.1}

The following claim follows from Proposition \ref{4.1} $f)$ and Claim \ref{8.2}.
\begin{lem}\label{8.3} 
$a)$ $\bo _{s,\pi}=\bo _\pi =0$ for all cuspidal representations $\pi$.

$b)$ $\bo _{s,x}=\int _{z\in S^1}\chi _{x^z}x(s)|dz|$  and
 $\bo _x=\int _{z\in S^1}\chi _{x^z}|dz|$ for all $x\in X$.

\end{lem}

\section{Proof of Theorem \ref{1.1}} 
The proof of Theorem \ref{1.1} using results on orbital integrals
 for the group $\ti G=\GL (2,F)$. We denote by $\ti \mcH$ the
 Hecke algebra for $\ti G$ and for any $r\geq 0$ define
 subalgebras $\ti \mcH _r\subset \ti \mcH$ as in \cite{Ka}. We
 fix a Haar measure on $\ti G$  identify $\ti \mcH$ 
with $\mcS (\ti G)$. For any $x\in X -X_2$ we denote by
$\ti \phi _x \in \ti \mcH$ the function supported on 
$\ti \Gamma _{d(x)}$ and equal to $\ti x(p(\gamma ))$ on $\ti \Gamma _{d(x)}$ and 
write $\ti \mu _x=\ti \phi _x \wt{dg}$.

The following result is immediate. 
\begin{cl}\label{9.1} 
 $\ti \mu_x\in \ti \mcH _{d(x)}$ for any $x\in X, \mu \in \mcH _x$.
\end{cl}
Let $\ti T\subset \ti G$ be a maximal elliptic torus, $\ti \mcT \subset M_2(F)$ the Lie algebra of $\ti T$. Define
$$ \ti \mcT _0=\ti \mcT \cap \varpi M_2(\mcO), \ti T_0=\ti T\cap \ti K_1.$$ 
It is clear that $\ti T_0$ is invariant under multiplications by $c\in U_1\subset \mcO ^\times$ and the projection $p:\ti G\to G$ induces a bijection $\ti T_0/U_1\to T_0, T_0=p(\ti T_0)$.
For any $a\in \ti \mcT _0$ we have $Id_2+a\in \ti T_0$.

The folowing Claim follows from \cite{Sh} and \cite{hc:admissible}.
\begin{cl}\label{9.2} 

$a)$. For any maximal elliptic torus $\ti T\subset\ti G$ there exist functions $\ti c_e, \ti c_\mcU$ on 
$\ti \mcT$ such that 
$$c_e(ca)=c_e(a), c_\mcU (ca)=\| c\ |^{-1}c_\mcU (a), c\in F^\times, a,ca\in \ti \mcT _0$$
 and for any $\ti \mu =\ti f\wt {dg}\in \ti \mcH $ there exists a neighborhood $Y_\mu$ of $0$ in $\ti \mcT$ such that $a\in Y_{\ti \mu}$ we have
$$(\star ){}I_\bO (\ti f)=\ti c_e(a)\delta (\ti f)+\ti c_\mcU (a)\bo (\ti f)$$
where 
$$\bO =(Id_2+a)^G, \delta (\ti f)=\ti f(e),\bo (\ti f)=\int _\mcU \ti f\nu .$$

Define functions $c_e,c_\mcU$ on $T_0$ by 
 $$c_e (g):= \ti c_e(g-Id_2), c_\mcU  (g):= \ti c_\mcU(g-Id_2).$$ Then 
 
$b)$ The functions $c_e, c_\mcU$  are invariant under multiplication by $cId_2\in \ti G, c\in U_1$.

 Therefore these functions define functions on $T_0$ which we also denote by  $c_e$ and $c_\mcU$.
\end{cl}

As follows from Appendix the results of  \cite{Ka} are applicable for all local fields $F$. So we have the following statement.

\begin{cl}\label{9.3} 

The equality $(\star)$ holds for all $a\in \varpi ^r\cap \ti \mcT$ if 
$\ti \mu =\ti f \wt {dg}\in \ti \mcH _r$.
\end{cl}
\begin{cor}\label{9.4} 
For any elliptic torus $T\subset G, t\in T\cap \Gamma _1, t\neq e$ 
and any $\mu =fdg \in \mcH _x, d(x)\geq 1$ we have

$$(\star ){}I_t (f)= c_e(t)\delta (f)+c_\mcU (t)\bo (f)$$
where $I_t:=I_{t^G}$.
\end{cor}
\begin{proof}   Apply the Corollary \ref{9.3} to $\ti \mu =p^\star (\mu)$.
\end{proof}
\begin{prop}\label{9.4a} 

For any $f\in \mcH _{cusp}$ we have
$$I_\bO (f)=\sum _\pi tr (\pi (f))d(\pi) $$
where $\pi$ runs through the set of equivalent classes of irreducible cuspidal representations of $G$.
\end{prop}
\begin{proof} Since $\mcH _{cusp}$ is spanned by matrix coefficients of irreducible cuspidal representations it is sufficient to chack the equality in the case when $f=m_\xi$
is a matrix coefficient  of an irreducible cuspidal representation $(\pi ,V),\xi \in V$ but in this case the equality follows from Proposition \ref{3.1} $c)$.
\end{proof}

\begin{thm}\label{9.5} 
For any elliptic torus $T\subset G$  and  any regular elliptic conjugacy class $\Omega =t^G\subset G, \bO \cap T_0\neq \emptyset$ and
any $\mu\in\mcH$, we have

$$(\star) I_\bO (f)=
\sum _\pi tr (\pi (f))d(\pi ,dg)+\sum _{x\in X|d(x)\leq d(\bO)}
(c_e(t)\delta _x(f)+c_u(t)\bo _x(f))$$
\end{thm}

\begin{proof}

As follows from Proposition \ref{9.1} the equality is true for
$f\in \mcS (G)_{cusp}$. So as follows from Lemma \ref{7.2}
it is sufficient to check the equality in the case of $f=\phi  _{n,x}\in\mcS (G)_x, x\in X, n\in \mZ$.

If $n\neq 0$ then $\bO$ does not intersect the support of $\phi _{n,x}$. On the other hand it follows from Lemma \ref{6.3} $b)$ 
and the Plancherel formula (Proposition \ref{8.1}) that 
 the right side of $\star$ also vanishes in this case.

 We see that for a proof of Theorem \ref{9.1} it is sufficient to check the equation $\star$ in the case $n=0$.

Since   $\delta _x(\ti \mu _c)=\delta (\mu _{0,x})$ and $\bo _x(\ti \mu _x)=\bo (\mu _{0,x})$
the equality $\star$ follows from Claim \ref{9.3} in the case when  $d(x)\leq d(\bO)$ .

On the other hand iff $d(x)>d(\bO)$, the equality follows from Corollary \ref{7.2}

\end{proof}

\section{The case of general groups} 
Let  $G$ be a split semisimple $F$-group. I fix a Haar measure $dg$ on $G$ and 
 often write $G$ instead of $G(F)$. Using the Haar mesure $dg$ one identifies $f\to fdg$ the space $\mcS (G)$ of locally 
constant $\mC$-valued compactly supported functions on $G$ with
the space $\mcH =\mcS (G)$  of locally constant $\mC$-valued compactly supported measures on $G$. The convolution defines 
 an algebra structure on  $\mcH$ and we define 
$$ \mcH _G:= \mcH / [\mcH , \mcH]$$
For $f\in \mcS (G)$ we denote by $f_G$ the image of $f$ in 
$ \mcH _G$. 

For any regular elliptic element 
$t\in G$ we  define a functional $I_t$ on $\mcH$ by

 $$ I_t(fdg)=\int _G f(gtg^{-1})dg$$
It is clear that $I_t$ does not depend on a choice a Haar measure and  depends only 
on the conjugacy class $c =c_t$ of $t$ in $G$. We write $I_c$ instead of $I_t$.

Let $C$ be the set of regular elliptic conjugacy classes of $G$.

There exists (see \cite{K}) a measure $dc$ on $C$ such that 
$$\int _Gf(g)dg=\int _ {c\in C}I_c(f)dc$$ 
for any $f\in \mcS (G)$ supported on the 
subset $G_e\subset G$ of regular elliptic elements.

Let $A(G)\subset \mcS (G)$ be the subset of functions such that 
$\int _\bO fd\bo=0$ for any regular non-elliptic conjugacy class 
$\bO$ where $\bo$ is an invariant measure on $\bO$.

We denote by 
$A_G$ the image of $A(G)$ in $\mcH_G$. For any $ a\in A_G$ we define a function $[a]$ on $C$ by 
$[a](c)= I_c([a]) $.
As follows from Theorem $F$ in \cite{K} for any $ [a], [b]\in A_G$ the scalar product 
$$< [a],[b] >:=\int _C[a]\ov {[b]}dc$$
is well defined.

Let $\mcZ$ be the Bernstein center of $G$. As follows 
from  Theorem $B$ in \cite{K} there exists a countable subset
 $S\subset Spec(\ mcZ)$ of characters and a decomposition
$$A_G=\oplus A_\bo ,\bo \in S$$
 of $A$ into a direct sum of finite dimensional subspaces such 
that $z\in \mcZ$ acts by $\bo (z)$ on $A_\bo ,\bo \in S$.  

As follows from \cite{K} the subspaces $A_\bo$ are mutually orthogonal and the restrictions of the form $<,>$ on the subspaces $A_\bo$ are positive definite. For any $\bo \in S$   we define 
a function $\phi _{\bo ,c}$ on $C$ with values in  complex-valued functions on $C$ by 
 $$\phi _{\bo ,c} =\sum _i \ov {[a_i]}(c)[a_i] $$
where $\{ a_i\}$ is an orthonormal basis of $A_\bo$.  We can  consider  $\phi _{\bo ,c}$ 
as a distribution on $\mcH$ where
$$\phi _{\bo ,c}(f)=\int _C \phi _{\bo ,c}I_c(f)$$ 
 
For a regular elliptic conjugacy class $c\subset G$ 
we define a functional $\alpha _c$ on $\mcH$ by 
$$\alpha _c(f):=I_c(f)-\sum _{\bo \in S}  <[f],\phi _\bo (c)>.$$

It is clear that for any $f\in \mcS (G)$ almost all summands of 
the sum $\sum _{\bo \in S}  <[f],\phi _\bo (c)>$ vanish.

If $char (F)=0$ then it follows from  \cite{hc:admissible} that $A(G)\subset \mcS (G)$ can be described as the subspace of functions $f$
 such that $tr (\pi (f))=0$ for all representations $\pi$ 
induced from an 
irreducible representations of a proper Levi subgroup of $G$. Therefore (see Theorem $B$ in \cite{K1})  
one can express the functional $\alpha _c$ in terms of traces of representations induced from an 
irreducible representations of a proper Levi subgroup of $G$. 

It would be interesting to find such an expression.

\appendix

\section{On homogeneity for characters of $\GL(n)$}
\begin{center}
\emph{Stephen DeBacker}
\end{center}

The following homogeneity result for $\GL(n,F)$, which is a refinement of the Harish-Chandra--Howe local character expansion~\cite{hc:admissible,howe:fourier},  is known to hold when the residue characteristic of $F$ is sufficiently large~\cite{debacker:homogeneity,waldspurger:finitude}.

Let $\gg$ denote the Lie algebra of $\GL(n,F)$,  let $\dborbit(0)$ denote the set of nilpotent orbits in $\gg$, and for $\dorbit \in \dborbit$ let $\hat{\mu}_\dorbit$ denote the function which represents the Fourier transform of the nilpotent orbital integral $\mu_\dorbit$.

\begin{thm} \label{thm:HCHowe}
Suppose $(\pi,V)$ is an irreducible smooth representation of $\GL(n,F)$ of depth $\rho(\pi)$.  If $\chi_\pi$ denotes the character of $\pi$,  then there exist complex constants $c_\dorbit(\pi)$, indexed by $\dorbit \in \dborbit(0)$, such that
$$\chi_\pi(1+X) = \sum_{\dorbit \in \dborbit(0)} c_\dorbit (\pi) \hat{\mu}_\dorbit (X)$$
for all regular semisimple $X \in \gg_{\rho(\pi)^+}$.
\end{thm}

In this appendix we (a) explain the notation that occurs in Theorem~\ref{thm:HCHowe} and its proof; (b) state a conjecture whose validity would imply  Theorem~\ref{thm:HCHowe} for $\GL(n,F)$ independent of the residue characteristic of $F$; and (c) prove this conjecture when $n = 2$.

\subsection*{Notation}

Recall that  $\mcO$
denotes the ring of integers of $F$,  and $\varpi$ denotes a
{uniformizer} so that $\cal{P} = \varpi \mcO$ where $\cal{P}$ is the prime
ideal.  We define $\cal{P}^m = \varpi^m \cdot \mcO$ for $m \in \Z$.
We fix an additive character $\Lambda$ of $F$ that is trivial on
$\cal{P}$ and not trivial on $\mcO$.

We realize $\GL(n,F)$ as the group of $n
\times n$ matrices with entries in $F$ having nonzero determinant.  We let $\Tsp$ denote the  subgroup consisting of diagonal matrices in $\GL(n,F)$.

 We realize $\gg$, the Lie algebra of $\GL(n,F)$, as
the algebra of $n \times n$ matrices with entries in the field
$F$ with the usual bracket operation.    The set of nilpotent matrices in $\gg$ is denoted by $\NN$.  The group $\GL(n,F)$ acts on $\NN$, and $\dorbit(0)$ denotes the corresponding finite set of nilpotent orbits.

For $i \in \Z$, we define the
{standard filtration lattices}
$\mathfrak{k}_i = \varpi^i \cdot \M_n(\mcO)$
of $\gg$ and  the {Iwahori filtration
lattices} $\mathfrak{b}_{i/n} = \mathfrak{b}_{i/n} = \{Y \in \gg \, | \, Y_{jk} \in
\varpi^{\lceil{\frac{j-k + i}{n}} \rceil} \cdot \mcO \}$.
Note that for all integers $i,j$, we have  $\varpi^j \cdot  \mathfrak{k}_i =
\mathfrak{k}_{i + j}$ and $\varpi^j \cdot \mathfrak{b}_{i/n} =
\mathfrak{b}_{j + \frac{i}{n}}$.
More concretely, for $n = 2$ we have
$$\mathfrak{k}_1 = \pmat{\cal{P}}{\cal{P}}{\cal{P}}{\cal{P}}$$
and
$$\mathfrak{b}_0 = \pmat{\mcO}{\mcO}{\cal{P}}{\mcO} \supset \mathfrak{b}_{1/2} =
\pmat{\cal{P}}{\mcO}{\cal{P}}{\cal{P}} \supset \mathfrak{b}_1 = \varpi \cdot
\mathfrak{b}_0.$$

Let $\BB$ denote the reduced Bruhat-Tits building of $\GL(n,F)$ and $\AA \subset \BB$ the apartment corresponding to $\Tsp$.  Let $C_0$ be an alcove in $\AA$.  The group $\GL(n,F)$ acts on $\BB$, and the orbit of every point in $\BB$ intersects the closure of $C_0$ at least once.  Moy and Prasad~\cite{moy-prasad:K-types,moy-prasad:jacquet} associated to each $x \in \BB$ and $r \in \R$  a lattice $\gg_{x,r}$ in $\gg$ and when $r \geq 0$ a compact open subgroup $G_{x,r}$ of $\GL(n,F)$.  For $\GL(n,F)$, we have $G_{x,0} = \gg_{x,0}^\times$ and $G_{x,r} = 1 + \gg_{x,r}$ for $r > 0$.   The Moy-Prasad lattices have the property that $\varpi \gg_{x,r} = \gg_{x,r+1}$ and $\gg_{x,s} \subset \gg_{x,r}$ for $s > r$.

 Since $g \cdot  \gg_{x,r}  = \gg_{gx,r}$, it is enough to understand the Moy-Prasad lattices for $x$ in the closure of $C_0$, and we do this for $\GL(2,F)$ in Figure~\ref{fig:gltwoaid}.  Here the apartment $\AA$ is identified with the horizontal axis and the vertical axis measures $r$.  The chamber $C_0$ has end points $x_0$ and $x'$.   The plane has been divided into polygonal regions by dotted lines and each polygonal region has been labeled by a lattice.    If $(x,r)$ lies in the interior of one of these polygonal regions, then $\gg_{x,r}$ is the lattice so labeled.
 If $(x,r)$ lies on a dotted line, then $\gg_{x,r}$  is given by the label of the polygonal region directly above the point $(x,r)$.  Moy and Prasad define
$$\gg_{x,r^+} = \bigcup_{s > r} \gg_{x,s}.$$
Note that $\gg_{x,r^+} \subset \gg_{x,r}$.
In  Figure~\ref{fig:gltwoaid} we have  $\gg_{x,r} = \gg_{x,r^+}$ unless $(x,r)$ lies on a dotted line, in which case $\gg_{x,r^+}$ is given by the label of the polygonal region directly below the point $(x,r)$.
  Similar notation is used for the Moy-Prasad subgroups.

For $r \in \R$ we define
$$\gg_{r} = \bigcup_{x \in \BB} \gg_{x,r} \, \, \, \text{ and } \, \, \,    \gg_{r^+} = \bigcup_{x \in \BB} \gg_{x,r^+}. $$
We have  $\gg_{r^+} \subset \gg_r$ and $\gg_r  \neq  \gg_{r^+}$ if and only if $n \cdot r \in \Z$.
From~\cite{adler-debacker:Kirillovn} we have
\begin{equation}  \label{equ:star1}
\gg_{r^+}= \bigcap_{x \in \BB}
(\gg_{x,r^+} + \NN)
\end{equation}
Note that $\NN \subset \gg_s$  for all $s \in \R$ and $\gg = \bigcup_s \gg_s$.

  For $f \in \mcS(\gg)$,  the space of compactly supported, complex valued, locally constant functions on $\gg$, we define $\hat{f}$, the {Fourier
  transform} of $f$, by the formula
$$\hat{f}(X) = \int_{\gg} f(Y) \cdot \Lambda( \tr (X \cdot Y)) \,  dY$$
for $X$ in $\gg$.  Here $dY$ is a fixed Haar measure on $\gg$.

If $L$ is a lattice in $\gg$, let $C_c(\gg/L)$ be the subspace of $\mcS(\gg)$ consisting of functions that are locally constant with respect to $L$.   If $L$ and $L'$ are lattices in $\gg$ with $L'  \subset L$, then $C(L/ L')$ denotes the subspace of $\mcS(\gg)$ consisting of functions supported in $L$ and locally constant with respect to $L'$.
Set
$$D_{r^+} = \sum_{x \in \BB} C_c(\gg/\gg_{x,r^+}).$$
From~\cite{adler-debacker:Kirillovn} we have
 $ D_{r^+} =
\widehat{\mcS(\gg_{-r})}$.

We denote by $J(\gg)$ the space of invariant distributions on $\gg$.  For example, if $\dorbit \in \dborbit(0)$, then $\mu_{\dorbit}$,  the corresponding orbital integral, lies in $J(\gg)$.   If $\omega$ is a closed, $G$ -invariant subset of $\gg$ (for example, $\NN$ or $\gg_{r^+}$), then $J(\omega)$ denotes the subspace of $J(\gg)$ consisting of invariant distributions with support in $\omega$.  If $\omega$ is compactly generated and $T \in J(\omega)$, then~\cite{hc:admissible,huntsinger:thesis} the distribution $\hat{T}$ defined by $\hat{T}(f) = T(\hat{f})$ for $f \in \mcS(\gg)$ is represented by a locally integrable function, which is also denoted $\hat{T}$,  on the set of regular semisimple elements in $\gg$.

\subsection*{A conjecture}
Fix an irreducible smooth representation $(\pi,V)$ of $\GL(n,F)$.  The depth of $\pi$, denoted by $\rho(\pi)$, is the smallest non-negative real number for which there exists $x \in \BB$ so that $V$ has non-trivial fixed vectors with respect to $G_{x,\rho(\pi)^+}$.   Choose $r$ such that $\gg_r \neq \gg_{r^+} = \gg_{-\rho(\pi)}$; such an $r$ must be of the form $k/n$ with $k \in \Z$.

For $x \in \BB$ and $s \leq r$, define
$$\tilde{J}_{x,s,r^+} = \{T \in J(\gg) \, | \, \text{for $f \in
C(\gg_{x,s}/\gg_{x,r^+})$, if $   \gg_{s^+}  \cap \supp(f)  =
\emptyset$, then $T(f) = 0$} \}.$$
Since $\NN \subset   \gg_{s^+}$, every invariant distribution supported on the set of nilpotent elements belongs to  $\tilde{J}_{x,s,r^+}$.
Set
$$\tilde{J}_{r^+} = \bigcap_{x \in \BB} \bigcap_{s \leq r}
\tilde{J}_{x,s,r^+}.$$
Note that $J(\NN) \subset \tilde{J}_{r^+}.$

For $T \in J(\gg)$ denote by  $\res_{D_{r^+}}T$  the restriction of $T$ to the space of functions $D_{r^+}$.   It is shown in~\cite[\S\S3.1--3.5]{debacker:homogeneity} that Theorem~\ref{thm:HCHowe}  follows from the following conjecture.

\begin{conj}\label{conj:one}
We have
$$\res_{D_{r^+}} \tilde{J}_{r^+} = \res_{D_{r^+}} J(\NN).$$
\end{conj}

For $z \in F^\times$, $f \in \mcS(\gg)$, and $T \in J(\gg)$, define $T_z(f) = T(f_z)$ and $f_z(X) = f(zX)$ for $X \in \gg$.  If $z$ has valuation $v$, then we have: (i) $T \in \tilde{J}_{r^+}$ if and only if $T_z \in \tilde{J}_{(r+v)^+}$; (ii) $f \in D_{r^+} $ if and only if $f_z \in D_{(r-v)^+}$; and (iii) $T' \in J(\NN)$ if and only if $T'_z \in J(\NN)$.  Thus, since $T(f) = T_z(f_{z\inv})$, it is enough to verify Conjecture~\ref{conj:one}  for $r \in \{ k/n \, | \, 0 \leq k < n\}$.

\subsection*{A proof for $\GL(2,F)$}
 Thanks to the
remarks at the end of the previous section, we only need to verify two
statements:
\begin{equation} \label{equ:depthzero}
\res_{D_{0^+}} \tilde{J}_{0^+} = \res_{D_{0^+}} J(\NN)
\end{equation}
and
\begin{equation} \label{equ:halfdepth}
\res_{D_{1/2^+}} \tilde{J}_{1/2^+} = \res_{D_{1/2^+}} J(\NN).
\end{equation}

We will prove Statement~(\ref{equ:halfdepth}).  A proof of Statement~(\ref{equ:depthzero}) may be carried out in a similar fashion (see also~\cite{debacker:lectures}).

   \subsubsection*{Descent and recovery}  Fix $T \in \tilde{J}(\gg_{1/2^+})$.  The goal of this section is  to show that $\res_{D_{1/2^+}}T$ is completely determined by
$\res_{C(\bb_{1/2}/\bb_{1})}T$, where $\bb_{1/2}$ and $\bb_1$ are Iwahori filtration lattices.

Fix $f \in D_{1/2^+}$. We write $f = \sum_i f_i$ with $f_i \in
C_c(\gg/\gg_{x_i,1/2^+})$ for some $x_i \in \BB$.  Since $T$ is linear,
without loss of generality we may assume that $f \in
C_c(\gg/\gg_{x,1/2^+})$ for some $x \in \BB$.  We can write
$$f = \sum_{\bar{Z} \in \gg/\gg_{x,1/2^+}} c_{\bar{Z}} \cdot [ Z +
\gg_{x,1/2^+} ]$$
where $ [ Z +
\gg_{x,1/2^+} ]$ denotes the characteristic function of the coset  $ Z +
\gg_{x,1/2^+}$ and all but finitely many of the complex constants $c_{\bar{Z}}$ are equal to  zero.
Again, since $T$ is linear, without loss of generality we may assume
that $f = [Z + \gg_{x,1/2^+}]$.

Choose $s \leq 1/2$ with the property that $Z + \gg_{x,1/2^+} \subset \gg_{x,s} \setminus \gg_{x,s^+}$.
By the definition of $\tilde{J}_{r^+}$ and Property~\ref{equ:star1}, we have $T(f) = 0$ if the support of $f$ does not intersect $\gg_{x,s^+} + \NN$.
That is, $T(f) = 0$ unless
$$(Z + \gg_{x,s^+}) \cap \NN \neq \emptyset.$$
So, without loss of generality, we may assume $Z = X + Y$ with $X \in \NN \cap (\gg_{x,s}\setminus \gg_{x,s^+})$ and $Y \in \gg_{x,s^+}$.

Up to conjugacy, we have two choices for $\gg_{x,1/2^+}$; it is either
$\mathfrak{k}_1$ or $\bb_{1}$.  In what follows, the reader is encouraged to
consult Figure~\ref{fig:gltwoaid}.
\begin{figure}[ht]
\begin{center}

\psfrag{r}[][]{$r$}
\psfrag{o}[][]{$1$}
\psfrag{z}[][]{$0$}
\psfrag{m}[][]{$-m$}
\psfrag{om}[][]{$1-m$}
\psfrag{rprr}[][]{$\pmat{\mcO}{\cal{P}}{\mcO}{\mcO}$}
\psfrag{rrrr}[][]{$\pmat{\mcO}{\mcO}{\mcO}{\mcO}$}
\psfrag{rrpr}[][]{$\pmat{\mcO}{\mcO}{\cal{P}}{\mcO}$}
\psfrag{rmpr}[][]{$\pmat{\mcO}{\cal{P}\inv}{\cal{P}}{\mcO}$}
\psfrag{rmtr}[][]{$\pmat{\mcO}{\cal{P}\inv}{\cal{P}^2}{\mcO}$}
\psfrag{pprp}[][]{$\pmat{\cal{P}}{\cal{P}}{\mcO}{\cal{P}}$}
\psfrag{pppp}[][]{$\pmat{\cal{P}}{\cal{P}}{\cal{P}}{\cal{P}}$}
\psfrag{prpp}[][]{$\pmat{\cal{P}}{\mcO}{\cal{P}}{\cal{P}}$}
\psfrag{prtp}[][]{$\pmat{\cal{P}}{\mcO}{\cal{P}^2}{\cal{P}}$}
\psfrag{pmtp}[][]{$\pmat{\cal{P}}{\cal{P}\inv}{\cal{P}^2}{\cal{P}}$}
\psfrag{ptpp}[][]{$\pmat{\cal{P}}{\cal{P}^2}{\cal{P}}{\cal{P}}$}
\psfrag{pptp}[][]{$\pmat{\cal{P}}{\cal{P}}{\cal{P}^2}{\cal{P}}$}
\psfrag{prthp}[][]{$\pmat{\cal{P}}{\mcO}{\cal{P}^3}{\cal{P}}$}
\psfrag{rrrrm}[][]{$\varpi^{-m} \cdot \pmat{\mcO}{\mcO}{\mcO}{\mcO}$}
\psfrag{rrprm}[][]{$\varpi^{-m} \cdot \pmat{\mcO}{\mcO}{\cal{P}}{\mcO}$}
\psfrag{rmprm}[][]{$\varpi^{-m} \cdot \pmat{\mcO}{\cal{P}\inv}{\cal{P}}{\mcO}$}
\psfrag{ppppm}[][]{$\varpi^{-m} \cdot \pmat{\cal{P}}{\cal{P}}{\cal{P}}{\cal{P}}$}
\psfrag{prppm}[][]{$\varpi^{-m} \cdot \pmat{\cal{P}}{\mcO}{\cal{P}}{\cal{P}}$}
\psfrag{prtpm}[][]{$\varpi^{-m} \cdot \pmat{\cal{P}}{\mcO}{\cal{P}^2}{\cal{P}}$}
\psfrag{pptpm}[][]{$\varpi^{-m} \cdot \pmat{\cal{P}}{\cal{P}}{\cal{P}^2}{\cal{P}}$}
\psfrag{x}[][]{$x_0$}
\psfrag{y}[][]{$$}
\psfrag{zp}[][]{$x'$}
\includegraphics[scale=.52]{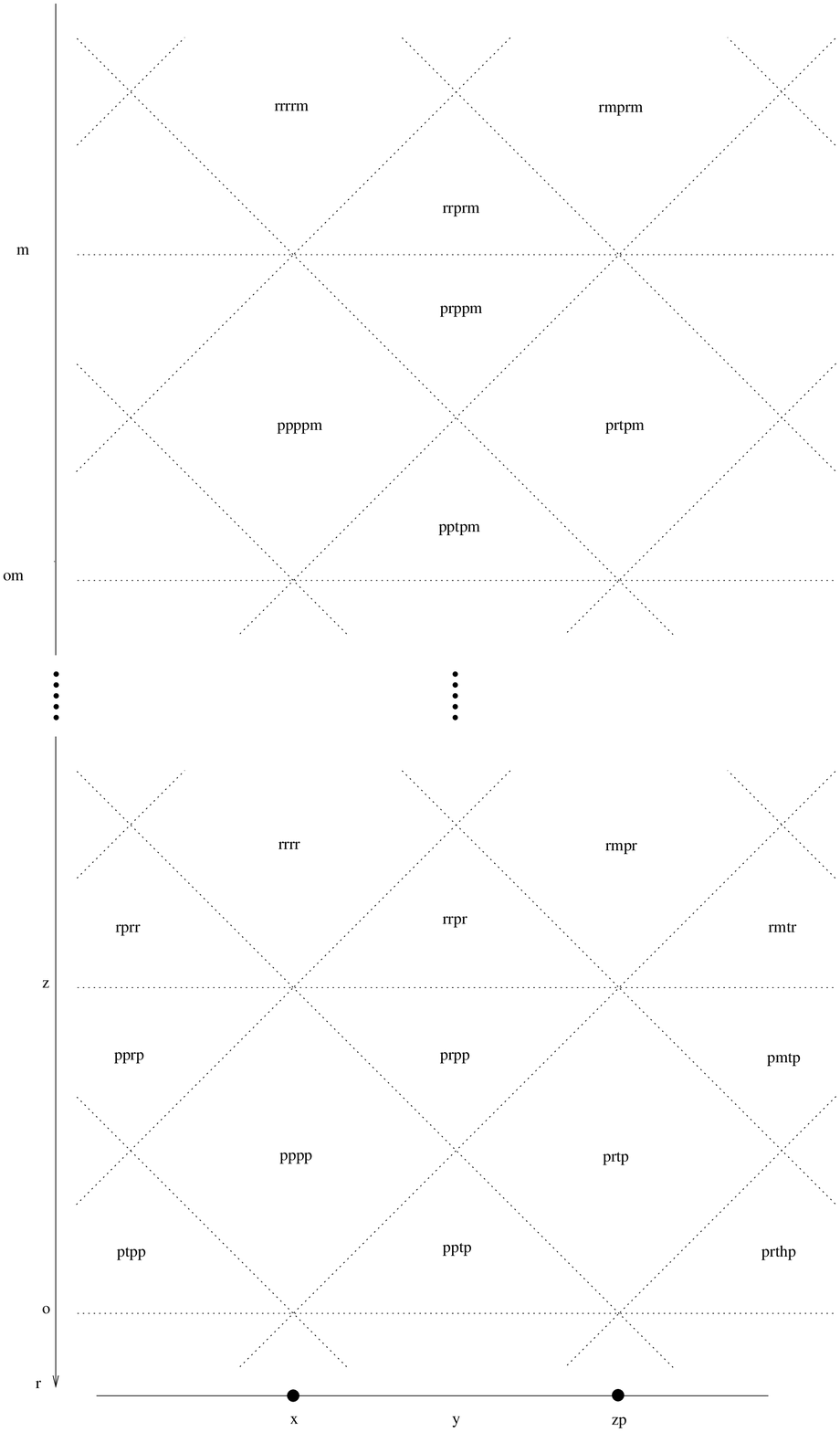}
\end{center}
\caption{$\gg_{x,r}$}
\label{fig:gltwoaid}
\end{figure}

We first examine the $ \gg_{x,1/2^+} = \bb_{1}$ case.  In this case, we are looking at
the coset $X + Y +  \gg_{y,1/2^+}$ where $y$ is the barycenter of $C_0$,  $X \in \NN \cap ( \gg_{y,s}\setminus \gg_{x,s^+})$, and $Z \in \gg_{x,s^+}$.  Since $ \NN \cap (\gg_{y,s}\setminus \gg_{x,s^+})$ is empty unless $s = -m+1/2$ for $m \in \Z_{\geq 0}$, we may assume $s$ has this form.
Since we are trying to show that $\res_{D_{1/2^+}}T$ is completely determined by
$\res_{C(\bb_{1/2}/\bb_{1})}T$, we may assume $m \geq 1$.
Since $T$ is a $G$-invariant distribution,  after conjugating by $\stab_{\GL(2,F)}(y) = \langle \pmat{0}{1}{\varpi}{0}  \rangle \ltimes \bb_0^\times$ we may assume that $X$ is
$$\pmat{0}{0}{\varpi^{(1-m)}u}{0}$$
with $m > 0$ and $u \in \mcO^\times$.   Let $\Tsp_m = \Tsp \cap G_{x_0,m}$.
We write
\begin{equation*}
\begin{split}
T([X + Y + \bb_1])
&= \dfrac{1}{q^2} \cdot \sum_{\bar{t} \in \Tsp_m / \Tsp_{m^+}} T([{t}(X + Y)t\inv +  \bb_1])\\
&=\dfrac{1}{q^2}  \cdot \sum_{\bar{\alpha}, \bar{\beta} \in \mcO/\cal{P}}  T([X + Y  + \pmat{0}{0} {\varpi u(\alpha - \beta)}{0} + \bb_1])\\
&= \dfrac{1}{q} \cdot T([X + Y +  \mathfrak{k}_1]).
\end{split}
\end{equation*}
Note that $X + Y +  \mathfrak{k}_1 \subset \varpi^{-m} \mathfrak{k}_1 = \gg_{x_0,s^+}$.
Thus,
we have expressed $T$ evaluated at $f = [X + Y  + \bb_1]$ in terms of $T$ evaluated at $f' = \dfrac{1}{q} \cdot [X + Y +  \mathfrak{k}_1]$ where $f' \in D_{1/2^+}$ has support closer to the origin with respect to the $x_0$ filtration than $f$ had with respect to the $y$ filtration.

We now examine the $ \gg_{x,1/2^+} = \mathfrak{k}_{1}$ case.  We may suppose that $s = -m$ for some $m  \in \Z_{\geq 0}$, so that  $X \in \NN \cap (\gg_{x_0, -m} \smallsetminus \gg_{x_0,(-m)^+})$ and $Y \in \gg_{x_0,(-m)^+}$.
Since $T$ is $G$-invariant, after conjugating by $\mathfrak{k}_0^\times$ we may assume that $X$ is
$$\pmat{0}{\varpi^{-m}u}{0}{0}$$
with $m \geq 0$ and $u \in \mcO^\times$.
We then have
$$T([X + Y +  \mathfrak{k}_1]) = \sum_{\bar{\alpha} \in \cal{P}/\cal{P}^2} T([ X + Y +  \pmat{0}{0}{\alpha}{0} + \bb_{1}])$$
Note that $X + Y +  \pmat{0}{0}{\alpha}{0} \in \gg_{y,s^+}$.  Thus,
we have expressed $T$ evaluated at $f = [X + Y  + \mathfrak{k}_1]$ in terms of $T$ evaluated at
$$f' = \sum_{\bar{\alpha} \in \cal{P}/\cal{P}^2} [ X + Y +  \pmat{0}{0}{\alpha}{0} + \bb_{1}]$$
 where $f' \in D_{0^+}$ has support closer to the origin with respect to the $y$ filtration than $f$ had with respect to the $x_0$ filtration.

To summarize, the point of {descent and recovery} is as follows.  We begin with a simple function $f \in C((\gg_{x,s}\smallsetminus \gg_{x,s^+})/\gg_{x,1/2^+})$ for some $x \in \BB$.  From this function, we find a point $y \in \BB$ and a function $f' \in C(\gg_{y,s^+}/\gg_{y,1/2^+})$ so that $T(f) = T(f')$.  After a finite number of steps, we will have shown that $T(f)$
is completely determined by $\res_{C(\bb_{1/2}/\bb_{1})}T$.

\subsubsection*{Counting}

From~\cite{hc:admissible} we know that the dimension of the complex vector space $\res_{D_{0^+}} J(\NN)$ is equal to  the number of nilpotent orbits.
Since $J(\NN) \subset \tilde{J}_{1/2^+}$, we have
$$2 = \dim_{\C}  \res_{D_{0^+}}J(\NN) \leq \dim_\C \res_{D_{0^+}} \tilde{J}_{1/2^+}.$$
 From our work above we have
 $$\dim_{\C} \res_{D_{0^+}}\tilde{J}_{1/2^+} = \dim_{\C}
\res_{C(\bb_{1/2}/\bb_{1})}\tilde{J}_{1/2^+}.$$
Consequently, we need only show that
 $\dim_{\C}  \res_{C(\bb_{1/2}/\bb_{1})}\tilde{J}_{1/2^+} \leq 2$.
Since $\gg_{1/2^+} \subset \gg_{x,1/2^+} + \NN$ for any $x \in \BB$, we have that for $T \in \tilde{J}_{1/2^+}$  the restriction of $T$  to $C(\bb_{1/2}/\bb_{1})$ is completely determined by
$$T([ \pmat{0}{1}{0}{0} + \bb_1])  \, \, \, \text{ and }
\, \, \,   T([\bb_1]).$$

\enlargethispage{\baselineskip}

\end{document}